\documentclass[11pt,reqno]{article}
\usepackage{amsmath,amssymb,graphicx,mathrsfs,amsopn,stmaryrd,amsthm,color,xcolor,bm}
\usepackage[colorlinks=true,citecolor=blue,linkcolor=blue,pagebackref]{hyperref}
\hypersetup{urlcolor=cyan,citecolor=[rgb]{0,0.5,0}}
\usepackage{indentfirst}
\usepackage[final]{showlabels}
\usepackage[lining]{libertine}
\usepackage[english]{babel}
\usepackage[mathscr]{euscript}

\usepackage{lipsum}
\let\OLDthebibliography\thebibliography
\renewcommand\thebibliography[1]{
	\OLDthebibliography{#1}
	\setlength{\parskip}{0pt}
	\setlength{\itemsep}{2pt} 
}
\usepackage{ytableau}
\usepackage{geometry}
\let\emptyset\varnothing
\usepackage{dsfont}
\numberwithin{equation}{section}
\newtheorem{theorem}{Theorem}[section]
\newtheorem{proposition}[theorem]{Proposition}
\newtheorem{lemma}[theorem]{Lemma}
\newtheorem{corollary}[theorem]{Corollary}
\newtheorem{conj}[theorem]{Conjecture}
\theoremstyle{definition} 
\newtheorem{eg}[theorem]{Example}

\theoremstyle{remark}
\newtheorem{rem}[theorem]{Remark}
\allowdisplaybreaks

\newcommand{\beq}{\begin{equation}}
\newcommand{\eeq}{\end{equation}}
\newcommand{\be}{\begin{equation*}}
\newcommand{\ee}{\end{equation*}}
\newcommand{\bs}{\boldsymbol}

\newcommand{\C}{\mathbb{C}}

\newcommand{\Z}{\mathbb{Z}}

\newcommand{\mc}{\mathcal}

\newcommand{\gl}{\mathfrak{gl}}
\newcommand{\h}{\mathfrak{h}}

\newcommand{\fkT}{\mathfrak{T}}

\newcommand{\End}{\mathrm{End}}

\newcommand{\pa}{\partial}
\newcommand{\tl}{\tilde}
\newcommand{\gge}{\geqslant}
\newcommand{\lle}{\leqslant}
\newcommand{\la}{\lambda}

\newcommand{\glMN}{\mathfrak{gl}_{m|n}}
\newcommand{\UglMN}{\mathrm{U}(\mathfrak{gl}_{m|n})}

\newcommand{\Yone}{\mathrm{Y}(\mathfrak{gl}_{1|1})}
\newcommand{\YglMN}{\mathrm{Y}(\mathfrak{gl}_{m|n})}

\newcommand{\ka}{\kappa}

\newcommand{\qedd}{\tag*{$\square$}}

\newcommand{\s}{{\bm s}}

\begin{document}
\pagestyle{myheadings}
\setcounter{page}{1}

\title{\sf A note on odd reflections of super Yangian and Bethe ansatz}

\author{Kang Lu}
\date{}

\maketitle
\vspace{-1.2cm}
\begin{center}
Department of Mathematics, University of Denver,\\
 2390 S. York St., Denver, CO 80208, USA
\\
{\sf Kang.Lu@du.edu}
\end{center}
\begin{abstract} 
In a recent paper \cite{M:2021}, Molev introduced analogues of the odd reflections for the super Yangian $\YglMN$ and obtained a transition rule for the change of highest weights when the parity sequence is altered. In this note, we reproduce the results from a different point of view and discuss their relations with the fermionic reproduction procedure of the XXX-type Bethe ansatz equations introduced in \cite{HLM:2019}. We give an algorithm that how the $q$-characters change under the odd reflections. We also take the chance to compute explicitly the $q$-characters of skew representations of $\YglMN$ for arbitrary parity sequences.

		\medskip 
		
\end{abstract}


\thispagestyle{empty}

\section{Introduction}
The super Yangians $\YglMN$ associated with general Lie superalgebra $\glMN$ were introduced back to \cite{Naz:1991} in the beginning of 1990s and their finite-dimensional irreducible representations were classified in \cite{Zh:1995,Zh:1996}. There is a considerable renewed interest to super Yangians \cite{Gow:2007,P:2016,Tsy:2020}, their representations \cite{Stu:2013,Jan:2016,LM:2020,L:2021a,M:2021}, and their relations to $\mathcal W$-algebras \cite{P:2021} and Bethe ansatz \cite{BR:2008,HLM:2019,LM:2020,LM:2019}. 

Comparing to the even case, a remarkable feature of Lie superalgebras is that there are non-conjugate Borel subalgebras which can be parameterized by the parity sequences for the case of $\glMN$. In the recent paper \cite{M:2021}, Molev introduced the odd reflections for the super Yangian $\YglMN$ and derived a transition rule for the change of the highest weight for a finite-dimensional irreducible $\YglMN$-module when the parity sequence is altered by switching two neighbors.

In this note, we reproduce the results of \cite{M:2021} from a different point of view by using Drinfeld-type generators. The results of this note were obtained in 2020 Summer after the completion of \cite{LM:2020}. Our motivation to study the odd reflections of super Yangian comes from the study of XXX Bethe ansatz equations (BAE for short), see \cite{HLM:2019}. Here, besides the parameters for roots of BAE, the BAE depend also on the parity sequence and the highest weight of the corresponding finite-dimensional irreducible $\YglMN$-module for that parity sequence $\s$, see \eqref{eq:bae}. The odd reflection for BAE is the \emph{fermionic reproduction procedure} introduced in \cite{HMVY:2019,HLM:2019} which given a solution of BAE associated to one parity sequence constructs a solution of BAE associated to another different parity sequence $\tl \s$ obtained by switching two neighbors. These statements are obtained for the case when the finite-dimensional irreducible $\YglMN$-modules are tensor products of evaluation modules since one knows explicitly the highest weights for other parity sequences by the known results from representation theory of $\glMN$. Then one expects that such a procedure should work for arbitrary finite-dimensional irreducible $\YglMN$-modules. To obtain that, it is natural to ask how does the highest weight changes under the odd reflections. However, one can also proceed on the other way around. Assuming these statements for arbitrary finite-dimensional irreducible $\YglMN$-modules and using the same procedure, one can obtain the formulas of the ``highest weights" of the module for the parity sequence $\tl \s$ in terms of the known highest weights of the same module for the parity sequence $\s$, see Section \ref{sec:bae}. Now it suffices to show this is indeed the ``highest weights" for the parity sequence $\tl \s$.

The idea to prove this transition rule for $\Yone$ is actually simple since all finite-dimensional irreducible $\Yone$-modules are tensor products of evaluation modules up to some one-dimensional $\Yone$-module. The highest weight vector for the parity sequence $\tl\s$ is simply the lowest weight vector for the parity sequence $\s$. Moreover, the action of $t_{ii}(u)$ on the lowest weight vector can be easily computed and hence we obtain the transition rule for the $\Yone$ case, see Theorem \ref{thm:11}. The case for $\YglMN$ is essentially reduced to the $\Yone$ case, see Theorem \ref{thm:mn}. Instead of working on RTT generators, we proceed this idea using the Drinfeld current generators, see Section \ref{sec:odd}.

Note that the subalgebra (called Gelfand-Tsetlin subalgebra) generated by the coefficients of the Cartan current generating series depends on the parity sequence. Therefore, for the same $\YglMN$-module, its $q$-character also depends on the choice of parity sequences. We give an algorithm of computing the $q$-characters after the odd reflections in Section \ref{sec:alg}.

We take the opportunity to study the skew representations for super Yangian with arbitrary parity sequence. We give explicitly their $q$-characters, see Theorem \ref{thm character} and show that skew representations as $\YglMN$-modules are independent of the parity sequences, see Proposition \ref{prop iso}, in Section \ref{sec skew}.

The paper is constructed as follows. We start by organizing known facts about the general linear Lie superalgebra $\glMN$ and the super Yangian $\YglMN$ for an arbitrary parity sequence in Section \ref{sec:super yangian}. In Section \ref{sec more}, we recall the quasi-determinant presentation of Drinfeld current generators and study the relations between Drinfeld current generators of different parity sequences. The transition rule for highest weights of representations of super Yangians and its relation to the Bethe ansatz are discussed in Section \ref{sec:odd}. Section \ref{sec skew} is devoted to the study of $\YglMN$-modules related to skew Young diagrams. 

\medskip

{\bf Acknowledgments.} The author thanks Evgeny Mukhin,  Yung-Ning Peng, and Vitaly Tarasov for stimulating discussions.

\section{Super Yangian $\YglMN$}\label{sec:super yangian}
\subsection{Lie superalgebra $\glMN$}\label{sec glmn}Throughout the paper, we work over $\C$. In this section, we recall the basics of the Lie superalgebra $\glMN$, see e.g. \cite{CW:2012} for more detail.

A \emph{vector superspace} $W = W_{\bar 0}\oplus W_{\bar 1}$ is a $\Z_2$-graded vector space. We call elements of $W_{\bar 0}$ \emph{even} and elements of
$W_{\bar 1}$ \emph{odd}. We write $|w|\in\{\bar 0,\bar 1\}$ for the parity of a homogeneous element $w\in W$. Set $(-1)^{\bar 0}=1$ and $(-1)^{\bar 1}=-1$.

Fix $m,n\in \Z_{\gge 0}$. Set $I:=\{1,2,\dots,m+n-1\}$ and $\bar I:=\{1,2,\dots,m+n\}$.  

Denote by $S_{m|n}$ the set of all sequences $\bm s=(s_{1},s_2,\dots,s_{m+n})$ where $s_i\in\{\pm1\}$ and $1$ occurs exactly $m$ times. Elements of $S_{m|n}$ are called \emph{parity sequences}. The parity sequence of the form $\bm{s_0}=(1,\dots,1,-1,\dots,-1)$ is the \emph{standard parity sequence}.  

Fix a parity sequence $\s\in S_{m|n}$ and define $|i|\in \Z_2$ for $i\in \bar I$ by $s_i=(-1)^{|i|}$.

The Lie superalgebra $\glMN^\s$ is generated by elements $e_{ij}^\s$, $i,j\in \bar I$, with the supercommutator relations
\[
[e^\s_{ij},e^\s_{kl}]=\delta_{jk}e^\s_{il}-(-1)^{(|i|+|j|)(|k|+|l|)}\delta_{il}e^\s_{kj},
\]
where the parity of $e_{ij}^\s$ is $|i|+|j|$. In the following, we shall drop the superscript $\s$ when there is no confusion.

Denote by $\UglMN$ the universal enveloping superalgebra of $\glMN$. The superalgebra $\UglMN$ is a Hopf superalgebra with the coproduct given by $\Delta(x)=1\otimes x+x\otimes 1$ for all $x\in \glMN$. 

The \emph{Cartan subalgebra $\h$} of $\glMN$ is spanned by $e_{ii}$, $i\in \bar I$. Let $\epsilon_i$, $i\in \bar I$, be a basis of $\h^*$ (the dual space of $\h$) such that $\epsilon_i(e_{jj})=\delta_{ij}$. There is a bilinear form $(\ ,\ )$ on $\h^*$ given by $(\epsilon_i,\epsilon_j)=s_i\delta_{ij}$. The \emph{root system $\bf{\Phi}$} is a subset of $\h^*$ given by
\[
{\bf \Phi}:=\{\epsilon_i-\epsilon_j~|~i,j\in \bar I \text{ and }i\ne j\}.
\]
We call a root $\epsilon_i-\epsilon_j$ \emph{even} (resp. \emph{odd}) if $|i|=|j|$ (resp. $|i|\ne |j|$).

Set $\alpha_i:=\epsilon_i-\epsilon_{i+1}$ for $i\in I$. Denote by ${\bf P}:=\oplus_{i\in \bar I}\Z \epsilon_i$, ${\bf Q}:=\oplus_{i\in I}\Z \alpha_i$, and ${\bf Q}_{\gge 0}:=\oplus_{i\in I}\Z_{\gge 0} \alpha_i$ the \emph{weight lattice}, the \emph{root lattice}, and the \emph{cone of positive roots}, respectively. Define a partial ordering $\gge$ on $\h^*$: $\mu\gge \nu$ if $\mu-\nu\in {\bf Q}_{\gge 0}$.

A module $M$ over a superalgebra $\mathcal A$ is a vector superspace $M$ with a homomorphism of superalgebras $\mathcal A\to \End(M)$. A $\glMN$-module is a module over $\mathrm{U}(\glMN)$. However, we shall not distinguish modules which are only differ by a parity.

For a $\glMN$-module $M$, we call a vector $v\in M$ \emph{singular} if $e_{ij}v=0$ for $1\lle i<j\lle m+n$. We call a nonzero vector $v\in M$ a \emph{singular vector of weight} $\mu$ if $v$ satisfies
\[
e_{ii}v=\mu(e_{ii})v,\quad e_{jk}v=0,
\]
for $i\in \bar I$ and $1\lle j< k\lle m+n$. Denote by $L(\mu)$ the irreducible $\glMN$-module generated by a 
singular vector of weight $\mu$. 

For a $\glMN$-module $M$, define the \emph{weight subspace of weight} $\mu$ by
\beq\label{eq:uweight-space}
(M)_{\mu}:=\{v\in M\ |\ e_{ii}v=\mu(e_{ii})v,\ i\in \bar I\}.
\eeq
For a $\glMN$-modules $M$ such that $(M)_{\mu}=0$ unless $\mu\in\bf P$, we say that $M$ is \emph{$\bf P$-graded}. 

Let $V:=\C^{m|n}$ be the vector superspace with a basis $v_i$, $i\in \bar I$, such that $|v_i|=|i|$. Let $E_{ij}\in\End(V)$ be the linear operators such that $E_{ij}v_k=\delta_{jk}v_i$. The map $\rho_V:\glMN\to \End(V),\ e_{ij}\mapsto E_{ij}$ defines a $\glMN$-module structure on $V$. As a $\glMN$-module, $V$ is isomorphic to $L({\epsilon_1})$. The vector $v_i$ has weight $\epsilon_i$. The highest weight vector is $v_1$ and the lowest weight vector is $v_{m+n}$. We call it the \emph{vector representation} of $\glMN$.

Fix $m',n'\in\Z_{\gge 0}$ and a parity sequence $\s'\in S_{m'|n'}$. Let $\s'\s$ be the parity sequence in $S_{m'+m|n'+n}$ obtained by concatenating $\s'$ and $\s$. Consider the Lie superalgebra $\gl_{m'|n'}^{\s'}$ and a larger Lie superalgebra $\gl_{m'+m|n'+n}^{\s'\s}$. 

Clearly, we have the embeddings of Lie superalgebras given by
\[
\gl_{m'|n'}^{\s'}\hookrightarrow\gl_{m'+m|n'+n}^{\s'\s},\quad e_{ij}^{\s'}\mapsto e_{ij}^{\s'\s},
\]
\[
\gl_{m|n}^{\s}\hookrightarrow\gl_{m'+m|n'+n}^{\s'\s},\quad e_{ij}^{\s}\mapsto e_{m'+n'+i,m'+n'+j}^{\s'\s}.
\]

\subsection{Hook partitions, skew Young diagrams, and polynomial modules}\label{sec:hook-p}
Let $\la=(\la_1\gge \la_2\gge \cdots)$ be a partition of $l$: $\la_i\in\Z_{\gge 0}$, $\la_i=0$ if $i\gg 0$, and $|\la|:=\sum_{i=1}^\infty \la_i=l$. We denote by $\la'$ the \emph{conjugate} of the partition $\la$. The number $\la_1'$ is the length of the partition $\la$, namely the number of nonzero parts of $\la$. Let $\mu=(\mu_1\gge \mu_2\gge \cdots)$ be another partition such that $\mu_i\lle \la_i$ for all $i=1,2,\dots$. Consider the \emph{skew Young diagram} $\la/ \mu$ which is defined as the set of pairs
\[
\{(i,j)\in\Z^2~|~i\gge 1,\ \la_i\gge j > \mu_i\}.
\]
When $\mu$ is the zero partition, then $\la/ \mu$ is the usual Young diagram corresponding to $\la$.

We use the standard representation of skew Young diagrams on the coordinate plane $\mathbb R^2$ with coordinates $(x,y)$. Here we use the convention that $x$ increases from north to south while $y$ increases from west to east. Moreover, the pair $(i,j)\in \la/ \mu$ is represented by the unit box whose south-eastern corner has coordinate $(i,j)\in\mathbb Z^2$. We also define the \emph{content} of the box corresponding to $(i,j)\in \la/ \mu$ by $c(i,j)=j-i$. 

Consider the alphabet set $A:=\{1,\dots,m,\underline 1,\dots,\underline n\}$ with a partial order given by $1<\dots<m$ and $\underline 1<\dots<\underline n$. Given a parity sequence $\s\in S_{m|n}$, define the elements $\mathfrak a_i\in A$, $i=1,\dots,m+n$, by the rule:
\begin{itemize}
\item  $\mathfrak a_i\in \{1,\dots,m\}$ if $s_i=1$ and $\mathfrak a_i\in \{\underline 1,\dots,\underline n\}$ if $s_i=-1$;
\item and $\mathfrak a_i<\mathfrak a_j$ if $s_i=s_j$ and $i<j$.
\end{itemize}
Extend the order $<$ to a total order $<^{\s}$ on $A$ by setting $\mathfrak a_i<^\s \mathfrak a_j$ if $i<j$.

\begin{eg}
If $\s=(1,-1,-1,1,1)$, then we have $1<^\s \underline 1<^\s\underline 2<^\s2<^\s3$.\qed
\end{eg}

A \textit{semi-standard Young $\s$-tableau} of shape $\la/\mu$ is the skew Young diagram $\la/\mu$ with an element from $A=\{1,\dots,m,\underline 1,\dots,\underline n\}$ inserted in each box such that the following conditions with respect to the order $<^\s$ are satisfied :
\begin{enumerate}
    \item the elements in boxes are weakly increasing along rows and columns;
    \item the elements from $\{1,2,\dots,m\}$ are strictly increasing along columns;
    \item the elements from $\{\underline 1,\dots,\underline n\}$ are strictly increasing along rows.
\end{enumerate}
For a semi-standard Young $\s$-tableau $\mathcal T$ of shape $\la/\mu$, set $\mathcal T(i,j)=k$ if the element $\mathfrak a_k$ is in the box representing the pair $(i,j)\in \la/\mu$. 
\begin{eg}
Let $\la=(5,3,3,3,3)$, $\mu=(3,3,2,2)$, $m=n=2$, and $\s=(1,-1,1,-1)$, then the skew Young diagram $\la/\mu$ has the shape as one of the following. 
$$
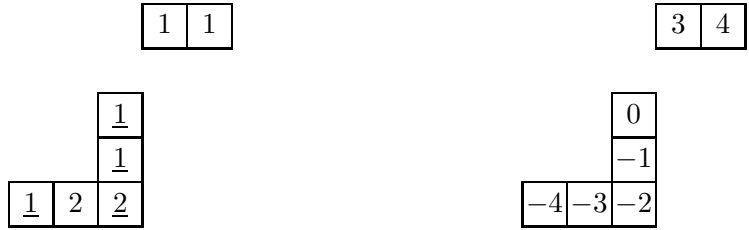
\begin{ytableau}
\none & \none & \none & 1 & 1\\
 \none & \none & \none\\
 \none & \none & \underline 1\\
 \none & \none & \underline 1\\
 \underline 1 & 2 & \underline 2	
\end{ytableau}
\qquad\qquad \qquad\qquad\qquad
\begin{ytableau}
\none & \none & \none & 3 & 4\\
 \none & \none & \none\\
 \none & \none & 0\\
 \none & \none & -1\\
 -4 & -3 & -2	
\end{ytableau}
$$
In the picture above, the left one is an example of a semi-standard Young $\s$-tableau of shape $\la/\mu$. We have $\mc T(1,4)=\mc T(1,5)=1$, $\mc T(3,3)=\mc
T(4,3)=\mc T(5,1)=2$, $\mc T(5,2)=3$, and $\mc T(5,3)=4$. In the right, we wrote in each box its content.  \qed
\end{eg}

Recall that $V=\C^{m|n}$ denotes the vector representation of $\gl_{m|n}$.
A $\gl_{m|n}$-module is called a \emph{polynomial module} if it is a submodule of $V^{\otimes l}$ for some $l\in\Z_{\gge 0}$. We call a partition $\la$ an $(m|n)$-{\it hook partition} if $\la_{m+1}\lle n$. Let $\mathscr P_l(m|n)$ be the set of all $(m|n)$-hook partitions of $l$ and $\mathscr P(m|n)$ the set of all $(m|n)$-hook partitions. In particular, $\mathscr P_l(m):=\mathscr P_l(m|0)$ is the set of all partitions of $l$ with length $\lle m$. It is well-known that irreducible polynomial $\gl_{m|n}$-modules are parameterized by $\mathscr P(m|n)$.

Given $\s\in S_{m|n}$, let $a_i$ (resp. $b_j$) be the number of $-1$'s (resp. $1$'s) appearing before the $i$-th $1$ (resp. $j$-th $-1$) in the parity sequence $\s$, where $1\lle i\lle m$ and $1\lle j\lle n$.

For $\la\in\mathscr P(m|n)$, define the $\gl_{m|n}$-weight $\la^\s$ by
\beq\label{eq natural}
\la^\s=\sum_{i=1}^m \max\{\la_i-a_i,0\}\epsilon_{a_i+i}+\sum_{j=1}^n \max\{\la_j'-b_j,0\}\epsilon_{b_j+j}.
\eeq

Let $\mc P\in \End(V \otimes V)$ be the super flip operator,
\beq\label{eq:P}
\mathcal P=\sum_{i,j\in \bar I} s_jE_{ij}\otimes E_{ji}.
\eeq
Let $\mathfrak S_l$ be the symmetric group permuting $\{1,2,\dots,l\}$. The symmetric group $\mathfrak S_l$ acts naturally on $V^{\otimes l}$, where the simple transposition $\sigma_k=(k,k+1)$ acts as 
\beq\label{eq perm}
\mathcal P^{(k,k+1)}=\sum_{i,j\in \bar I} s_j E_{ij}^{(k)}E_{ji}^{(k+1)}\in \End(V^{\otimes l}),
\eeq
where we use the standard notation
\[
E_{ij}^{(k)}=1^{\otimes (k-1)}\otimes E_{ij}\otimes 1^{\otimes (l -k)}\in \End(V^{\otimes l}),\qquad k=1,\dots,l.
\]

Let $\mathcal S(\la)$ be the finite-dimensional irreducible representation of $\mathfrak S_l$ corresponding to the partition $\la$.
\begin{theorem}[Schur-Sergeev duality \cite{Ser:1985}]\label{thm schur}
The $\mathfrak S_l$-action and $\gl_{m|n}$-action on $V^{\otimes l}$ commute. Moreover, as a $\mathrm{U}(\glMN)\otimes \C[\mathfrak S_l]$-module, we have
\[
V^{\otimes l}\cong\bigoplus_{\la\in\mathscr P_l(m|n)}L(\la^\s)\otimes \mathcal S(\la).\qedd
\]	
\end{theorem}

\subsection{Super Yangian $\YglMN$}\label{sec rtt}
We recall the definition of the super Yangian $\YglMN$ from \cite{Naz:1991}. 

The super Yangian $\YglMN$ (write as $\mathrm{Y}(\gl_{m|n}^\s)$ if we want to indicate the dependence on $\s$) is the $\Z_2$-graded unital associative algebra over $\C$ with generators $\{t_{ij}^{(r)}\ |\ i,j\in \bar I, \, r\gge 1\}$ and defining relations
\beq\label{eq:comm-generators}
[t_{ij}^{(r)},t_{kl}^{(s)}]=(-1)^{|i||j|+|i||k|+|j||k|}\sum_{a=0}^{\min(r,s)-1}(t_{kj}^{(a)}t_{il}^{(r+s-1-a)}-t_{kj}^{(r+s-1-a)}t_{il}^{(a)}),
\eeq
where the generators $t_{ij}^{(r)}$ have parities $|i|+|j|$. We call $t_{ij}^{(r)}$ \emph{RTT generators} of $\YglMN$.

The super Yangian $\YglMN$ has the RTT presentation as follows. Define the rational R-matrix $R(u)\in \End(V \otimes  V)$ by $R(u)=1-\mathcal P/u$, where $\mc P$ is defined in \eqref{eq:P}. The rational R-matrix satisfies the quantum Yang-Baxter equation
\beq\label{eq yang-baxter}
R_{12}(u_1-u_2)R_{13}(u_1-u_3)R_{23}(u_2-u_3)=R_{23}(u_2-u_3)R_{13}(u_1-u_3)R_{12}(u_1-u_2).
\eeq
Define the generating series
\[
t_{ij}(u)=\delta_{ij}+\sum_{k=1}^\infty t_{ij}^{(k)}u^{-k}
\]
and the operator $T(u)\in \End(V)\otimes\YglMN[[u^{-1}]] $,
\[
T(u)=\sum_{i,j\in \bar I} (-1)^{|i||j|+|j|}E_{ij}\otimes t_{ij}(u).
\]
Denote by
\beq\label{eq matrix notation}
T_k(u)=\sum_{i,j\in \bar I} (-1)^{|i||j|+|j|}E_{ij}^{(k)}\otimes t_{ij}(u)\in\End(V^{\otimes l}) \otimes\YglMN[[u^{-1}]],\quad k=1,2.
\eeq
Then defining relations \eqref{eq:comm-generators} can be written as
\[
R(u_1-u_2)T_1(u_1)T_2(u_2)=T_2(u_2)T_1(u_1)R(u_1-u_2)\in \End(V^{\otimes 2}) \otimes\YglMN[[u^{-1}]].
\]
In terms of generating series, defining relations \eqref{eq:comm-generators} are equivalent to
\beq\label{eq:comm-series}
(u_1-u_2)[t_{ij}(u_1),t_{kl}(u_2)]=(-1)^{|i||j|+|i||k|+|j||k|}(t_{kj}(u_1)t_{il}(u_2)-t_{kj}(u_2)t_{il}(u_1)).
\eeq

The super Yangian $\YglMN$ is a Hopf superalgebra with coproduct given by
\beq\label{eq Hopf}
\Delta: t_{ij}(u)\mapsto \sum_{k\in \bar I} t_{ik}(u)\otimes t_{kj}(u).
\eeq

For $z\in\C$ there exists an isomorphism of Hopf superalgebras,
\begin{align}
&\tau_z:\YglMN\to\YglMN, && t_{ij}(u)\mapsto t_{ij}(u-z).\label{eq tau z}
\end{align}

The universal enveloping superalgebra $\mathrm U(\glMN)$ is a Hopf subalgebra of $\YglMN$ via the embedding $e_{ij}\mapsto s_it_{ij}^{(1)}$. The left inverse of this embedding is the \emph{evaluation homomorphism} $\pi_{m|n}: \YglMN\to \UglMN$ given by
\beq\label{eq:evaluation-map}
\pi_{m|n}: t_{ij}(u)\mapsto \delta_{ij}+s_ie_{ij}u^{-1}.
\eeq

The evaluation homomorphism is an algebra homomorphism but not a Hopf algebra homomorphism.
For any $\glMN$-module $M$, it is naturally a $\YglMN$-module obtained by pulling back $M$ through the evaluation homomorphism $\pi_{m|n}$. We denote the corresponding $\YglMN$-module by the same letter $M$ and call it an \emph{evaluation module}.

\subsection{Gauss decomposition}\label{sec:Gauss}
The Gauss decomposition of $\YglMN$, see \cite{Gow:2007,P:2016,Tsy:2020}, gives generating series
\[
e_{ij}(u)=\sum_{r\gge 1}e_{ij,r}u^{-r},\quad f_{ji}(u)=\sum_{r\gge 1}f_{ji,r}u^{-r},\quad d_k(u)=1+\sum_{r\gge 1}d_{k,r}u^{-r},
\]
where $1\lle i< j\lle m+n$ and $k\in \bar I$, such that
\begin{align*}
t_{ii}(u)&=d_i(u)+\sum_{k<i}f_{ik}(u)d_k(u)e_{ki}(u),\\
t_{ij}(u)&=d_i(u)e_{ij}(u)+\sum_{k<i}f_{ik}(u)d_k(u)e_{kj}(u),\\
t_{ji}(u)&=f_{ji}(u)d_i(u)+\sum_{k<i}f_{jk}(u)d_k(u)e_{ki}(u).
\end{align*}

For $i\in I$ and $k\in \bar I$, let
\[
e_{i}(u)=e_{i,i+1}(u)=\sum_{r\gge 1}e_{i,r}u^{-r},\quad f_{i}(u)=f_{i+1,i}(u)=\sum_{r\gge 1}f_{i,r}u^{-r},
\]
\[
d_{k}'(u)=(d_k(u))^{-1}=1+\sum_{r\gge 1}d_{k,r}'u^{-r}.
\]
We use the convention $d_{k,0}=d_{k,0}=1$.

The parities of $e_{ij,r}$ and $f_{ji,r}$ are the same as that of $t_{ij}^{(r)}$ while all $d_{k,r}$ and $d_{k,r}'$ are even. The super Yangian $\YglMN$ is generated by $e_{i,r}$, $f_{i,r}$, $d_{k,r}$, $d_{k,r}'$, where $i\in I$ and $k\in \bar I$, and $r\gge 1$. We call these generators the \emph{Drinfeld current generators} of $\YglMN$. The full defining relations are described in \cite{Gow:2007,P:2016,Tsy:2020}. Here we only write down the following relations. Let $\phi_i(u)=d_i'(u)d_{i+1}(u)=1+\sum_{r\gge 1}\phi_{i,r}u^{-r}$. Then one has $[d_{i,r}, d_{j,s}]=0$,
\begin{align}
&[d_{i,r}, e_{j,s}]=(\epsilon_i,\alpha_j) \sum_{t=0}^{r-1} d_{i,t} e_{j,r+s-1-t},\quad [d_{i,r}, f_{j,s}]=-(\epsilon_i,\alpha_j)\sum_{t=0}^{r-1} f_{j,r+s-1-t} d_{i,t}, \label{eq com de}\\
&[e_{j,r}, f_{k,s}]=-s_{j+1}\delta_{jk}\sum_{t=0}^{r+s-1} d_{j,t}' d_{j+1,r+s-1-t}=-s_{j+1}\delta_{jk}\phi_{j,r+s-1}.\label{eq com ef}
\end{align}
Moreover, if $s_i\ne s_{i+1}$ for some $i\in I$, then $[e_{i,r},e_{i,s}]=[f_{i,r},f_{i,s}]=0$.

Let $\mathrm{Y}_{m|n}^+$, $\mathrm{Y}_{m|n}^-$, and $\mathrm{Y}_{m|n}^0$ be the subalgebras of $\YglMN$ generated by coefficients of the series $e_i(u)$, $f_i(u)$, and $d_j(u)$, respectively. It is known from \cite{Gow:2007,P:2016,Tsy:2020} that 
\beq\label{eq:tri}
\YglMN\cong \mathrm{Y}_{m|n}^-\otimes \mathrm{Y}_{m|n}^0 \otimes \mathrm{Y}_{m|n}^+
\eeq 
as vector spaces and $d_i^{(r)}$ are algebraically free generators of $\mathrm{Y}_{m|n}^0$.

\subsection{Highest and lowest $\ell$-weight representations}\label{sec high rep}
Recall $\mathcal B=1+u^{-1}\C[[u^{-1}]]$ and set $\mathfrak B:=\mathcal B^{\bar I}$.
We call an element $\bm\zeta\in \mathfrak B$ an \emph{$\ell$-weight}. We write $\ell$-weights in the form $\bm\zeta=(\zeta_i(u))_{i\in \bar I}$, where $\zeta_i(u)\in \mathcal B$ for all $i\in \bar I$.

Clearly $\mathfrak B$ is an abelian group with respect to the point-wise multiplication of the tuples. Let $\Z[\mathfrak B]$ be the group ring of $\mathfrak B$ whose elements are finite $\Z$-linear combinations of the form $\sum a_{\bm \zeta}[\bm\zeta]$, where $a_{\bm\zeta}\in \Z$.

Let $M$ be a $\YglMN$-module. We say that a nonzero $\Z_2$-homogeneous vector $v\in M$ is \emph{of $\ell$-weight $\bm \zeta$} if $d_{i}(u)v=\zeta_i(u)v$ for $i\in \bar I$. We say that a vector $v\in M$ of $\ell$-weight $\bs\zeta$ is a \emph{highest (resp. lowest) $\ell$-weight vector of $\ell$-weight $\bs\zeta$} if $e_{ij}(u)v=0$ (resp. $f_{ji}(u)v=0$) for all $1\lle i< j\lle m+n$ and $v$ is homogeneous. The module $M$ is called a \emph{highest (resp. lowest) $\ell$-weight module of $\ell$-weight $\bm \zeta$} if $M$ is generated by a highest (resp. lowest) $\ell$-weight vector of $\ell$-weight $\bm \zeta$. When we would like to stress the dependence on $\s$, we shall write $\ell^\s$-weight.

In general, $\ell$-weight vectors do not need to be eigenvectors of $t_{ii}(u)$. However, from the Gauss decomposition one can deduce that $v$ is a highest $\ell$-weight vector of $\ell$-weight $\bs \zeta $ if and only if
\beq\label{eq:t-singular}
t_{ij}(u)v=0,\quad t_{kk}(u)v=\zeta_k(u)v,\quad 1\lle i<j\lle m+n,\ k\in \bar I.
\eeq
Note that similar formulas do not hold for a lowest $\ell$-weight vector $v$ of $\ell$-weight $\bm \zeta$. 

A $\YglMN$-module $M$ is called \emph{thin} if $M$ has a basis consisting of $\ell$-weight vectors with distinct $\ell$-weights.

Let $v$ and $v'$ be highest $\ell$-weight vectors of $\ell$-weights $\bs \zeta$ and $\bs \vartheta$, respectively. Then, by \eqref{eq:t-singular} and \eqref{eq Hopf}, 
we have
\[
t_{ij}(u)(v\otimes v')=0,\quad t_{kk}(u)(v\otimes v')=\zeta_k(u)\vartheta_k(u)(v\otimes v'), \quad 1\lle i<j\lle m+n,\ k\in \bar I.
\]
Hence $v\otimes v'$ is a highest $\ell$-weight vector of $\ell$-weight $\bs \zeta \bs \vartheta$. In particular, we have
\[
e_{i}(u)(v\otimes v')=0,\quad d_{j}(u)(v\otimes v')=\zeta_j(u)\vartheta_j(u)(v\otimes v'), \quad i\in I,\ j\in \bar I.
\]

Every finite-dimensional irreducible $\YglMN$-module is a highest $\ell$-weight module. Let $\bm\zeta \in \mathfrak B$ be an $\ell$-weight. There exists a unique irreducible highest $\ell$-weight $\YglMN$-module of highest $\ell$-weight $\bm\zeta$. We denote it by $L(\bm \zeta)$. The criterion for $L(\bm \zeta)$ to be finite-dimensional when $\s$ is the standard parity sequence is as follows.

\begin{theorem}[\cite{Zh:1996}]\label{thm:Zh}
When $\s$ is the standard parity sequence, then the irreducible $\YglMN$-module $L(\bm \zeta)$ is finite-dimensional if and only if there exist monic polynomials $g_i(u)$, $i\in \bar I$, such that
\[
\frac{\zeta_i(u)}{\zeta_{i+1}(u)}=\frac{g_i(u+s_i)}{g_i(u)},\quad \frac{\zeta_m(u)}{\zeta_{m+1}(u)}=\frac{g_m(u)}{g_{m+n}(u)}, \quad i\in I,\ i\ne m,	
\]
and $\deg g_m=\deg g_{m+n}$. \qed
\end{theorem}

We finish this section with the following  proposition. Define the length function $\iota:{\bf Q}_{\gge 0}\to \Z_{\gge 0}$ by $\iota(\sum_{i\in I}n_i\alpha_i)=\sum_{i\in I}n_i$.
\begin{proposition}\label{prop:cop}
For $i\in I$, $j\in \bar I$, $k\in \Z_{>0}$, we have
\begin{align}
&\Delta({d_j^{(k)}}) -\sum_{l=0}^{k}d_{j}^{(l)}\otimes d_j^{(k-l)}\in \sum_{\iota(\mu)>0} (\YglMN)_{\mu}\otimes (\YglMN)_{-\mu},\label{eq copro d}\\
&\Delta(e_i^{(k)})-1\otimes e_i^{(k)}\in \sum_{\iota(\mu)>0} (\YglMN)_{\mu}\otimes (\YglMN)_{\alpha_i-\mu},\label{eq copro e} \\
&\Delta(f_i^{(k)})-f_i^{(k)}\otimes 1\in \sum_{\iota(\mu)>0} (\YglMN)_{\mu-\alpha_i}\otimes (\YglMN)_{-\mu} \label{eq copro f}.
\end{align}
\end{proposition}
\begin{proof}
The proof is similar to the case when $\s$ is the standard parity sequence in \cite[Proposition 2.7]{LM:2020}.
\end{proof}

Regard $\UglMN$ as a subalgebra of $\YglMN$, then we have $d_{i,1}=s_ie_{ii}$. In particular, one assigns a usual $\glMN$-weight to an $\ell$-weight via the map $\varpi:\mathfrak B\to \h^*,\ \bm\zeta\mapsto \varpi(\bm\zeta)$ by $\varpi(\bm\zeta)(e_{ii})=s_i\zeta_{i,1},$ where $\zeta_{i,1}$ is the coefficient of $u^{-1}$ in $\zeta_i(u)$.  

Given a $\YglMN$-module, consider it as a $\glMN$-module and its $\glMN$-weight subspaces $(M)_{\mu}$, see \eqref{eq:uweight-space}. The following lemma is straightforward from the equalities $[d_{i,r}, d_{j,s}]=0$, \eqref{eq com de}, and \eqref{eq com ef}.
\begin{lemma}\label{lem:wt-change}
We have
\[
e_{i,r}(M)_{\mu}\subset (M)_{\mu+\alpha_i},\quad f_{i,r}(M)_{\mu}\subset (M)_{\mu-\alpha_i},\quad d_{j,r}(M)_{\mu}\subset (M)_{\mu},
\]
for $i\in I$, $j\in \bar I$, and $r\in\Z_{>0}$.\qed
\end{lemma}

\subsection{Category $\mathcal C$ and $q$-character map}\label{sec:q-char}
Let $\mathcal C$ be the category of finite-dimensional $\YglMN$-modules. The category $\mathcal C$ is abelian and monoidal. Let $M\in \mathcal C$ be a finite-dimensional $\YglMN$-module and $\bm\zeta\in \mathfrak B$ an $\ell$-weight. Let
\[
\zeta_i(u)=1+\sum_{j=1}^\infty \zeta_{i,j}u^{-j},\qquad \zeta_i^{(j)}\in \C.
\]
Denote by $M_{\bm \zeta}$ the \emph{generalized $\ell$-weight subspace} corresponding to the $\ell$-weight $\bm\zeta$,
\[
M_{\bm\zeta}:=\{v\in M~|~  (d_{i,j}-\zeta_{i,j})^{\dim M} v=0 \text{ for all }i\in \bar I, \ j\in \Z_{>0}\}.
\]

For a finite-dimensional $\YglMN$-module $M$, define the $q$-{\it character} (or \emph{Yangian character}) of $M$ by the element 
\[
\chi(M):=\sum_{\bm\zeta\in \mathfrak B}\dim(M_{\bm \zeta})[\bm\zeta]\in \Z[\mathfrak B].
\]
This definition is a straightforward generalization of the even case, see \cite{Knight:1995}. It is also called the {\it Gelfand-Tsetlin character}, see \cite[Definition 8.5.7]{Molev:2007}, since the commutative subalgebra $\mathrm Y_{m|n}^0$ generated by coefficients of $d_i(u)$ for all $i\in \bar I$ is usually called the {\it Gelfand-Tsetlin subalgebra} of $\YglMN$. Note that the Gelfand-Tsetlin subalgebra do depend on the parity sequence $\s$ (even up to conjugation on the level of representations), see e.g. \cite[Example 3.1]{L:2021a} where a $\Yone$-module is thin for the standard parity sequence but not thin for the other one.

Let $\mathscr Rep(\mathcal C)$ be the Grothendieck ring of $\mathcal C$, then $\chi$ induces a $\Z$-linear map from $\mathscr Rep(\mathcal C)$ to $\Z[\mathfrak B]$.

\begin{lemma}\label{lem chi morphism}
The map $\chi:\mathscr Rep(\mathcal C)\to \Z[\mathfrak B]$ is an injective ring homomorphism. In particular, the Grothendieck ring $\mathscr Rep(\mathcal C)$ is commutative.
\end{lemma}
\begin{proof}
The proof of the first statement is similar to that of \cite[Lemma 2.8]{LM:2020}. The second statement is \cite[Corollary 2.9]{LM:2020}. 
\end{proof}

\section{More on super Yangians}\label{sec more}
\subsection{Quasi-determinant}\label{sec:quasi}
We shall also need the quasi-determinant presentation, see \cite{GGRW:2005}, of Drinfeld current generating series in terms of RTT generating series. Let $X$ be a square matrix over a ring with identity such that its inverse matrix $X^{-1}$ exists, 
and such that its $(j,i)$-th entry is an invertible element of the ring.  Then the $(i,j)$-th
\emph{quasi-determinant} of $X$ is defined by the formula
{\small \begin{equation*}
|X|_{ij} = \left((X^{-1})_{ji}\right)^{-1} =: \left| \begin{array}{ccccc} x_{11} & \cdots & x_{1j} & \cdots & x_{1n}\\
&\cdots & & \cdots&\\
x_{i1} &\cdots &\boxed{x_{ij}} & \cdots & x_{in}\\
& \cdots& &\cdots & \\
x_{n1} & \cdots & x_{nj}& \cdots & x_{nn}
\end{array} \right|.
\end{equation*}
}
By \cite[Theorem 4.96]{GGRW:2005}, one has
{
\begin{eqnarray*}
d_i (u) &=& \left| \begin{array}{cccc} t_{11}(u) &\cdots &t_{1,i-1}(u) &t_{1i}(u) \\
\vdots &\ddots & &\vdots \\
t_{i1}(u) &\cdots &t_{i,i-1}(u) &\boxed{t_{ii}(u)}
\end{array} \right|, \\
e_{ij}(u) &=& d_i (u)^{-1} \left| \begin{array}{cccc} t_{11}(u) &\cdots &t_{1,i-1}(u) & t_{1j}(u) \\
\vdots &\ddots &\vdots & \vdots \\
t_{i-1,i}(u) &\cdots &t_{i-1,i-1}(u) & t_{i-1,j}(u)\\
t_{i1}(u) &\cdots &t_{i,i-1}(u) &\boxed{t_{ij}(u)}
\end{array} \right|,\\
f_{ji}(u) &=& \left| \begin{array}{cccc} t_{11}(u) &\cdots &t_{1, i-1}(u) & t_{1i}(u) \\
\vdots &\ddots &\vdots &\vdots \\
t_{i-1,1}(u) &\cdots &t_{i-1,i-1}(u) &t_{i-1,i}(u)\\
t_{ji}(u) &\cdots &t_{j, i-1}(u) &\boxed{t_{ji}(u)} 
\end{array} \right|
d_{i}(u)^{-1}.
\end{eqnarray*}
}

\subsection{Homomorphisms between super Yangians}
Fix $m',n'\in\Z_{\gge 0}$ and a parity sequence $\s'\in S_{m'|n'}$. Let $\s'\s$ be the parity sequence in $S_{m'+m|n'+n}$ by concatenating $\s'$ and $\s$. Consider the Lie superalgebra $\gl_{m'|n'}^{\s'}$ and a larger Lie superalgebra $\gl_{m'+m|n'+n}^{\s'\s}$. In the following, we shall drop parity sequence when there is no confusion.

Let $\eta_{m|n}$ be the automorphism of $\YglMN$ given by
\beq\label{eq iso antipode}
\eta_{m|n}: T(u)\mapsto T(-u)^{-1}.
\eeq
Let $\varphi_{m'|n'}:\YglMN\to \mathrm Y(\gl_{m'+m|n'+n})$ be the embedding given by
\[
\varphi_{m'|n'}: t_{ij}(u)\mapsto t_{m'+n'+i,m'+n'+j}(u).
\]
Let $\psi_{m'|n'}:\YglMN\to \mathrm Y(\gl_{m'+m|n'+n})$ be the injective homomorphism given by
\beq\label{eq:psi}
\psi_{m'|n'}:=\eta_{m'+m|n'+n}\circ \varphi_{m'|n'}\circ \eta_{m|n}.
\eeq
Then, for any $1 \lle i,j \lle m+n$, see \cite[Lemma 4.2]{P:2016}, we have:
\begin{equation}\label{eq:psi t}
\psi_{m'|n'} (t_{ij}(u)) = \left| \begin{array}{cccc} t_{11}(u) &\cdots &t_{1,m'+n'}(u) &t_{1, m'+n'+j}(u)\\
\vdots &\ddots &\vdots &\vdots \\
t_{m'+n',1}(u) &\cdots &t_{m'+n',m'+n'}(u) &t_{m'+n', m'+n'+j}(u)\\
t_{m'+n'+i, 1}(u) &\cdots &t_{m'+n'+i,m'+n'}(u) &\boxed{t_{m'+n'+i, m'+n'+j}(u)}
\end{array} \right|.
\end{equation}
The following lemma can be found in \cite{P:2016}.

\begin{lemma}[\cite{Gow:2007,P:2016}]\label{eq psi map}
We have
\[
\psi_{m'|n'}(d_i(u))=d_{m'+n'+i}(u),\quad \psi_{m'|n'}(e_i(u))=e_{m'+n'+i}(u),\quad \psi_{m'|n'}(f_i(u))=f_{m'+n'+i}(u).
\]
\end{lemma}

\subsection{Isomorphisms between super Yangians of different parity sequences}\label{sec:id}
Let $\s$ be a parity sequence in $S_{m|n}$. Fix $i\in I$ and assume that $s_i\ne s_{i+1}$. Let $\tl\s$ be the parity sequence in $S_{m|n}$ obtained from $\s$ by switching $s_i$ and $s_{i+1}$. Let $\sigma$ be the simple permutation $(i,i+1)$ in the symmetric group $\mathfrak S_{m+n}$. The super Yangians $\mathrm{Y}(\gl_{m|n}^\s)$ and $\mathrm{Y}(\gl_{m|n}^{\tl\s})$ are isomorphic via the assignment $t_{ab}^\s(u)\mapsto t_{\sigma(a)\sigma(b)}^{\tl\s}(u)$, $1\lle a,b\lle m+n$. We shall identify super Yangians with different parity sequences using these isomorphisms.

\begin{eg}
Let $\s$ be a parity sequence in $S_{1|1}$ and $\tl\s$ the other parity sequence in $S_{1|1}$. The super Yangians $\mathrm{Y}(\gl_{1|1}^\s)$ and $\mathrm{Y}(\gl_{1|1}^{\tl\s})$ are isomorphic via the assignment $t_{ij}^\s(u)\mapsto t_{3-i,3-j}^{\tl\s}(u)$. Identifying these two Yangians using this isomorphism, we have
\beq\label{eq:current11}
\begin{split}
&d_1^{\tl\s}(u)e^{\tl\s}(u)=t_{12}^{\tl\s}(u)=t_{21}^\s(u)=f^\s(u)d_1^\s(u),\\
&f^{\tl\s}(u)d_1^{\tl\s}(u)=t_{21}^{\tl\s}(u)=t_{12}^\s(u)=d_1^\s(u)e^\s(u),\\
&	d_1^\s(u)=t_{11}^{\s}(u)=t_{22}^{\tl\s}(u)
=d_2^{\tl\s}+f^{\tl\s}(u)d_1^{\tl\s}(u)e^{\tl\s}(u)\\
&d_1^{\tl\s}(u-\tl s_1)(d_2^{\tl\s}(u-\tl s_1))^{-1}=(d_1^{\s}(u))^{-1}d_2^{\s}(u),
\end{split}
\eeq
where the last equality is obtained from either \cite[Section 4.3]{LM:2019} or \cite[Proposition 3.6]{HM:2020}. Throughout the paper, we shall drop the index $i$ if $i$ only takes values from $I$ and $m=n=1$.\qed
\end{eg}

\begin{eg}\label{eg:mns}
For the $\glMN$ case, by formulas in Section \ref{sec:quasi}, we have $d_j^\s(u)=d_j^{\tl\s}(u)$, for $j\ne i,i+1$, $x_j^\s(u)=x_j^{\tl\s}(u)$, for $x=e,f$ and $j\ne i$. Moreover, using the map $\psi$ defined as in \eqref{eq:psi} with suitable indices and choices of superalgebras, we have
\beq\label{eq:currentmn}
\begin{split}
&d_i^{\tl\s}(u)e_i^{\tl\s}(u)=f_i^\s(u)d_i^\s(u),\\
&f_i^{\tl\s}(u)d_i^{\tl\s}(u)=d_i^\s(u)e_i^\s(u),\\
&	d_i^\s(u)=d_{i+1}^{\tl\s}+f_i^{\tl\s}(u)d_i^{\tl\s}(u)e_i^{\tl\s}(u)\\
&d_i^{\tl\s}(u-\tl s_i)(d_{i+1}^{\tl\s}(u-\tl s_i))^{-1}=(d_i^{\s}(u))^{-1}d_{i+1}^{\s}(u).
\end{split}
\eeq
where the equalities are proved by using the quasi-determinant presentation of current generators, Lemma \ref{eq psi map}, and \cite[Lemma 4.2]{P:2016}, cf. \eqref{eq:current11}. For the the last equality, see also \cite[Proposition 3.6]{HM:2020}.	\qed
\end{eg}

\section{Odd reflections in Drinfeld current presentation}

\subsection{Odd reflections of super Yangian}\label{sec:odd}
\subsubsection{The case of $\gl_{1|1}$}\label{sec:odd11}
We recall some basic facts about representations of $\Yone$ from \cite{Zh:1995}. Let $\bm \zeta^\s=(\zeta^\s_1(u),\zeta^\s_2(u))$, where $\zeta^\s_1(u),\zeta^\s_2(u)\in\mathcal B=1+u^{-1}\C[[u^{-1}]]$. Denote $L(\bm\zeta^\s)$ the irreducible $\Yone$-module generated by a highest $\ell^\s$-weight vector $v^\s$ of $\ell^\s$-weight $\bm \zeta^\s$. It is known from \cite[Theorem 4]{Zh:1995} (or Theorem \ref{thm:Zh}) that $L(\bm\zeta^\s)$ is finite-dimensional if and only if
\beq\label{eq last 9}
\frac{\zeta^\s_1(u)}{\zeta^\s_2(u)}=\frac{\varphi^\s(u)}{\psi^\s(u)},
\eeq
where $\varphi^\s$ and $\psi^\s$ are relatively prime monic polynomials in $u$ of the same degree. Moreover, if $\deg\varphi^\s=k$, then $\dim L(\bm\zeta^\s)=2^k$ and a basis of $L(\bm\zeta^\s)$ in terms of RTT generators is also given in \cite[Equation (18)]{Zh:1995}. Here we reformulate it using the Drinfeld current generators.

Let
\beq\label{eq last 8}
\varphi^\s(u)=\prod_{i=1}^k(u+a_i)=u^k+\sum_{i=1}^k \gamma_i u^{k-i},\quad \psi^\s(u)=\prod_{j=1}^k(u-b_j)=u^k+\sum_{j=1}^k \eta_j u^{k-j},
\eeq
where $a_i,\gamma_i,b_j,\eta_j\in\C$. Then $a_i+b_j\ne 0$ for all $1\lle i,j\lle k$.

\begin{lemma}\label{lem:basis}
The set $\{f_{i_1}^\s f_{i_2}^\s \cdots f_{i_j}^\s v^\s,\ 1\lle i_1<\cdots<i_j\lle k,\ 0\lle j\lle k\}$ form a basis of $L(\bm\zeta^{\s})$.	
\end{lemma}
\begin{proof}
Since $\dim L(\bm\zeta^\s)=2^k$, it suffices to show that this set spans $L(\bm\zeta^\s)$. Note that the generators $f_i^\s $ anti-commute, by the triangular decomposition \eqref{eq:tri}, it reduces to show that $f_{k+j}^\s v^\s$ for $j\in \Z_{>0}$ can be expressed as a linear combination of $f_1^\s v^\s,\dots,f_k^\s v^\s$. 

Recall that $\phi^\s(u)=(d_1^\s(u))^{-1}d_2^\s(u)$. Let $\zeta^\s_2(u)/\zeta^\s_1(u)=1+\sum_{r\gge 1}\theta_ru^{-r}$, where $\theta_r\in \C$. Then it follows from \eqref{eq last 9} and \eqref{eq last 8} that
\[
\theta_{k+j}=-\sum_{i=1}^k\gamma_i\theta_{k+j-i},
\]
which in turn implies by \eqref{eq com ef} that
\begin{align*}
-s_2[e_i^\s ,\gamma_kf_j^\s +\cdots+\gamma_1f_{k+j-1}^\s +f_{k+j}^\s ]v^\s =\sum_{r=1}^k\gamma_{k+1-r}\theta_{i+j+r-2}v^\s +\theta_{k+i+j-1}v^\s =0,
\end{align*}
for all $i,j\in\Z_{>0}$.

Since $v^\s $ is a highest $\ell^\s $-weight vector, we conclude that $e_i^\s (\gamma_kf_j^\s +\cdots+\gamma_1f_{k+j-1}^\s +f_{k+j}^\s )v^\s=0$ for all $i,j\in\Z_{>0}$. Fix $j\in \Z_{>0}$. We must have $(\gamma_kf_j^\s +\cdots+\gamma_1f_{k+j-1}^\s +f_{k+j}^\s )v^\s= 0$. Otherwise, by the triangular decomposition \eqref{eq:tri} and Lemma \ref{lem:wt-change}, the submodule generated by $(\gamma_kf_j^\s +\cdots+\gamma_1f_{k+j-1}^\s +f_{k+j}^\s )v^\s$ does not contain $v^\s $ and hence is proper and nonzero which contradicts the irreducibility of $L(\bm\zeta^\s)$. Therefore $f_{k+j}^\s v^\s=-(\gamma_kf_j^\s +\cdots+\gamma_1f_{k+j-1}^\s )v^\s$. Now it is clear by induction that $f_{k+j}^\s v^\s $ for $j\in \Z_{>0}$ can be expressed as a linear combination of $f_1^\s v^\s,\dots,f_k^\s v^\s$. 
\end{proof}

For a subset $J$ of $\{1,\dots,k\}$, set $\varphi^\s_{J}=\prod_{i\in J } (u+a_i)$. By convention, $\varphi^\s_{\emptyset}=1$.

\begin{lemma}[{\cite[Lemma 3.8]{L:2021a}}]
\label{lem:char-1-1}
We have 
$$
\chi(L(\bm\zeta^\s))=\sum_{J\subset \{1,\dots,k\}}[\bm\zeta^\s]\cdot\left[\Big(\frac{\varphi^\s_{J}(u-s_1)}{\varphi^\s_{J}(u)},\frac{\varphi^\s_{J}(u-s_1)}{\varphi^\s_{J}(u)}\Big)\right],
$$	
where the summation is over all subsets of $\{1,\dots,k\}$.
\end{lemma}

Let $\tl\s$ be the other parity sequence in $S_{1|1}$. The following results are obtained in \cite{M:2021}. Here we prove them using a different approach. 

\begin{theorem}\label{thm:11}
The vector $v^{\tl\s}:=f_1^\s \cdots f_k^\s v^\s \in L(\bm\zeta^\s)$ is a highest $\ell^{\tl\s}$-weight vector of $\ell^{\tl\s}$-weight
\[
\Big(\zeta_2^\s(u)\frac{\psi^\s(u-s_1)}{\psi^\s(u)},\zeta_1^\s(u)\frac{\varphi^\s(u-s_1)}{\varphi^\s(u)}\Big).
\]
\end{theorem}
\begin{proof}
By Lemma \ref{lem:char-1-1}, the generalized $\ell^\s$-weight subspace corresponding to the $\ell^\s$-weight $$\Big(\zeta_1^\s(u)\frac{\varphi^\s(u-s_1)}{\varphi^\s(u)},\zeta_2^\s(u)\frac{\varphi^\s(u-s_1)}{\varphi^\s(u)}\Big)$$ has dimensional one and a nonzero vector from it must be a joint eigenvector of $d_1^\s(u)$ and $d_2^\s(u)$. Comparing the usual $\glMN$-weights for elements of the basis in Lemma \ref{lem:basis}, one has that $v^{\s'}$ is inside of this subspace. Therefore, we have
\[
d_1^\s(u)v^{\tl\s}=\zeta_1^\s(u)\frac{\varphi^\s(u-s_1)}{\varphi^\s(u)}v^{\tl\s},\quad d_2^\s(u)v^{\tl\s}=\zeta_2^\s(u)\frac{\varphi^\s(u-s_1)}{\varphi^\s(u)}v^{\tl\s}.
\]
On one hand, by Lemma \ref{lem:basis} and its proof, we have $f^\s(u)v^{\tl\s}=0$ and hence, by \eqref{eq:current11},
\[
e^{\tl\s}(u)v^{\tl\s}=(d_1^{\tl\s}(u))^{-1}f^\s(u)d_1^\s(u)v^{\tl\s}=0.
\]On the other hand, again by \eqref{eq:current11}, we have
\[
d_2^{\tl\s}(u)v^{\tl\s}=\big(d_1^{\s}(u)-f^{\tl\s}(u)f^{\s}(u)d_1^{\s}(u)\big)v^{\tl\s}=d_1^{\s}(u)v^{\tl\s}=\zeta_1^\s(u)\frac{\varphi^\s(u-s_1)}{\varphi^\s(u)}v^{\tl\s},
\]
while the action of $d_1^{\tl\s}(u)$ on $v^{\tl\s}$ can be computed by the last equality of \eqref{eq:current11}.
\end{proof}
\begin{rem}\label{rem:irr}
	By \cite[Theorem 4.9]{Jan:2016}, if the $\Yone$-module $L(\bm\zeta^\s)$ is not finite-dimensional (namely $\zeta^\s_1(u)/\zeta^\s_2(u)$ is not an expansion of a rational function at $\infty$), then $L(\bm\zeta^\s)$ is indeed the Verma-type module. Therefore, by PBW theorem of the super Yangian \cite{P:2016,Tsy:2020}, a basis of $L(\bm\zeta^\s)$ is given by $$
	\{f_{i_1}^\s f_{i_2}^\s \cdots f_{i_j}^\s v^\s,\ 1\lle i_1<\cdots<i_j,\ 0\lle j\}.$$ In particular, it is easy to see that there is no highest $\ell^{\tl\s}$-weight vector in $L(\bm\zeta^\s)$. Hence the odd reflection is not applicable for $L(\bm\zeta^\s)$ if $L(\bm\zeta^\s)$ is not finite-dimensional.\qed
\end{rem}

\begin{rem}
The result of Theorem \ref{thm:11} was implicitly used in \cite[Remark 3.11]{L:2021a}.\qed
\end{rem}

\subsubsection{The case of $\gl_{m|n}$}
Let $\s$ be a parity sequence in $S_{m|n}$. Fix $i\in I$ and assume that $s_i\ne s_{i+1}$. Let $\tl\s$ be the parity sequence in $S_{m|n}$ obtained from $\s$ by switching $s_i$ and $s_{i+1}$.

Let $\bm\zeta^{\s}=(\zeta^\s_i(u))_{i\in \bar I}$ be an $\ell^\s$-weight and $L(\bm\zeta^{\s})$ the corresponding irreducible highest $\ell^\s$-weight module with a highest $\ell^\s$-weight vector $v^\s$. Suppose further that $\zeta^\s_i(u)/\zeta^\s_{i+1}(u)$ is a series as a rational function expanded at $\infty$, see Remark \ref{rem:irr}. Let
\beq\label{eq:rat-cond}
\frac{\zeta^\s_i(u)}{\zeta^\s_{i+1}(u)}=\frac{\varphi^\s_i(u)}{\psi^\s_{i}(u)},
\eeq
where $\varphi^\s_i(u)$ and $\psi^\s_{i}(u)$ are relatively prime monic polynomials of degree $k$ for some $k\in\Z_{\gge 0}$.

Now we generalize Theorem \ref{thm:11} to the general $\glMN$ case.

\begin{theorem}\label{thm:mn}
The vector $v^{\tl\s}:=f_{i,1}^\s\cdots f_{i,k}^\s v^\s\in L(\bm\zeta^{\s})$ is a  highest $\ell^{\tl\s}$-weight vector of $\ell^{\tl\s}$-weight $\bm \zeta^{\tl\s}=(\zeta^{\tl\s}_j(u))_{j\in \bar I}$, where $\zeta^{\tl\s}_j(u)=\zeta^{\s}_j(u)$ for $j\ne i,i+1$ and
\beq\label{eq:tranmn}
\zeta_i^{\tl\s}(u)=\zeta_{i+1}^\s(u)\frac{\psi_i^\s(u-s_i)}{\psi_i^\s(u)},\qquad \zeta_{i+1}^{\tl\s}(u)=\zeta_i^\s(u)\frac{\varphi_i^\s(u-s_i)}{\varphi_i^\s(u)}.
\eeq
\end{theorem}
\begin{proof}
The subalgebra generated by the coefficients of $d_i^\s(u)$, $d_{i+1}^\s(u)$, $e_i^\s(u)$, $f_i^\s(u)$ is isomorphic to the super Yangian $\mathrm{Y}(\gl_{1|1}^{\bar\s})$ where $\bar\s:=(s_i,s_{i+1})$, by e.g. Lemma \ref{eq psi map} and \eqref{eq:psi t}. 

Consider the $\mathrm{Y}(\gl_{1|1}^{\bar\s})$-module $M$ generated by $v^\s$. Then $M$ is a highest $\ell^{\bar\s}$-weight $\Yone$-module with highest $\ell^{\bar\s}$-weight $\bar{\bm\zeta}^{\bar\s}=(\zeta_i^\s(u),\zeta_{i+1}^\s(u))$. Since $L(\bm\zeta^{\s})$ is irreducible, by the same reasoning as in Lemma \ref{lem:basis}, we have a similar basis of $M$ described as in Lemma \ref{lem:basis} and $\dim M= 2^k=\dim L(\bar{\bm\zeta}^{\bar\s})$. Therefore, repeating the proof of Theorem \ref{thm:11} and using \eqref{eq:currentmn}, we obtain $e_i^{\tl\s}(u)v^{\tl\s}=0$ and the equality \eqref{eq:tranmn}.

If $j\ne i,i+1$, then $d_j^{\tl\s}(u)=d_j^{\s}(u)$ commutes with $f_{i,r}^\s$ for $r\in\Z_{>0}$ and hence $d_j^{\tl\s}(u)v^{\tl\s}=\zeta_j^\s(u)v^{\tl\s}$. It remains to show that $e_j^{\tl\s}(u)v^{\tl\s}=0$ for $j\ne i$ which follows from the fact that $e_j^{\tl\s}(u)=e_j^{\s}(u)$ and $[e_j^{\s}(u),f_i^\s(w)]=0$. 
\end{proof}

\begin{rem}
Note that the condition that $\zeta^\s_i(u)/\zeta^\s_{i+1}(u)$ is a series as a rational function expanded at $\infty$ is necessary in order to perform odd reflection in this direction, see Remark \ref{rem:irr}.\qed
\end{rem}

\begin{rem}
Instead of repeating the proof of Theorem \ref{thm:11}, one can also apply Theorem \ref{thm:11} directly if we observe the follow fact. The isomorphism defined by the assignments $$d_1^{\bar\s}(u),d_2^{\bar\s}(u),e^{\bar\s}(u),f^{\bar\s}(u)\mapsto d_i^\s(u),d_{i+1}^\s(u),e_i^\s(u),f_i^\s(u),$$also satisfies the assignments $d_1^{\bar\s'}(u),d_2^{\bar\s'}(u),e^{\bar\s'}(u),f^{\bar\s'}(u)\mapsto d_i^{\tl\s}(u),d_{i+1}^{\tl\s}(u),e_i^{\tl\s}(u),f_i^{\tl\s}(u)$, where $\bar\s'=(s_{i+1},s_i)$. Indeed, this fact can be proved using Lemma \ref{eq psi map} and \eqref{eq:psi t}, cf. Example \ref{eg:mns}.\qed
\end{rem}

\subsection{Relation with Bethe ansatz}\label{sec:bae}
In this section, we explain our motivation to study odd reflections of super Yangian in relation with the fermionic reproduction procedure (also called B\"{a}cklund transformations) of Bethe ansatz equations following \cite{HLM:2019,LM:2020}. 

Following \cite{MR:2014}, define the \emph{universal rational difference operator} $\mathfrak D(u)$ as the quantum Berezinian of $1-T(u)e^{-\pa_u}$, see \cite{Naz:1991} for quantum Berezinian for standard parity sequence and \cite{HM:2020} for arbitrary parity sequence. Then
\beq\label{eq diff trans}
\mathfrak D(u)=\sum_{r=0}^\infty (-1)^r \fkT_r(u)e^{-r\pa_u},
\eeq
see e.g. \cite[Theorem 2.13]{MR:2014}, where $\fkT_r(u)$ are series in $u^{-1}$ with coefficients in $\YglMN$. Note that with the identification of super Yangians with different parity sequences as in Section \ref{sec:odd}, the quantum Berezinian is independent of the parity sequence, see \cite[Proposition 3.6]{HM:2020}. We call $\fkT_r(u)$ {\it transfer matrices}. The transfer matrices commute with each other.

Let $L(\bm\zeta)$ be a finite-dimensional irreducible $\YglMN$-module of highest $\ell$-weight $\bm \zeta=(\zeta_i(u))_{i\in \bar I}$. One of the main problems in Bethe ansatz is to find the spectra of the universal rational difference operator $\mathfrak D(u,\tau)$ (joint spectra of transfer matrices) acting on the space $L(\bm\zeta)$.

Let $\bm l=(l_i)_{i\in I}$ be a sequence of non-negative integers. Let $ \bm t=(t_{j}^{(i)})$, $i\in I$, $j=1,\dots,l_i$,  be a sequence of complex numbers. Define monic polynomials 
$$
y_i(u)=\prod_{j=1}^{l_i}(u-t_j^{(i)}),
$$
and set $\bm y=(y_i)_{i\in I}$.

The {\it Bethe ansatz equation} associated to $\bm \zeta$, $\bm s$ is a system of algebraic equations in $\bm y$ (actually in $\bm t$) given by
\beq\label{eq:bae}
\frac{\zeta_i(t_j^{(i)})}{\zeta_{i+1}(t_j^{(i)})}\frac{y_{i-1}(t_j^{(i)}+s_i)}{y_{i-1}(t_j^{(i)})}\frac{y_{i}(t_j^{(i)}-s_i)}{y_{i}(t_j^{(i)}+s_{i+1})}\frac{y_{i+1}(t_j^{(i)})}{y_{i+1}(t_j^{(i)}-s_{i+1})}=1,
\eeq
where $i\in I$ and $j=1,\dots,l_i$. It is known that when the Bethe ansatz equation is satisfied, one can construct the {\it Bethe vector} $\mathbb B_{\bm l}(\bm y)\in L(\bm\zeta)$ (depending on parameters $\bm t$) which is shown to be an eigenvector (if it is nonzero) of the first transfer matrix $\fkT_1(u)$, see \cite{BR:2008}. Indeed, the Bethe vector is also expected to be an eigenvector of all transfer matrices. 

Let $y_0(u)=y_{m+n}(u)=1$. Define a rational difference operator $\mathfrak D(u,\bm \zeta,\bm y,\s)$ by
\beq\label{eq bae eigenvalue}
\mathfrak D(u,\bm \zeta,\bs y)=\mathop{\overrightarrow\prod}\limits_{1\lle i\lle m+n}\Big(1-\zeta_i(u)\cdot\frac{y_{i-1}(u+s_i)y_i(u-s_i)}{y_{i-1}(u)y_i(u)}\, e^{-\pa_u}\Big)^{s_i}.
\eeq

\begin{conj}[{\cite[Conjecture 5.15]{LM:2020}}]\label{thm bae}
If $\bm t$ satisfies the Bethe ansatz equation \eqref{eq:bae}, then we have
\[
\mathfrak D(u)\, \mathbb B_{\bm l}(\bm y)=\mathfrak D(u,\bm \zeta,\bs y,\s)\, \mathbb B_{\bm l}(\bm y).\qedd
\]
\end{conj}

\subsubsection{The case of $\gl_{1|1}$}
Let $\s$ be a parity sequence in $S_{1|1}$ and $\tl\s$ the other parity sequence in $S_{1|1}$. Recall $\bm\zeta^\s=(\zeta^\s_1(u),\zeta^\s_2(u))$, the  finite-dimensional irreducible $\mathrm{Y}(\gl_{1|1}^\s)$-module $L(\bm\zeta^\s)$, and the coprime monic polynomials $\varphi^\s(u)$ and $\psi^\s(u)$ such that $\zeta_1^\s(u)/\zeta^\s_2(u)=\varphi^\s(u)/\psi^\s(u)$.

The Bethe ansatz equation in this case is reduced to $\zeta_1^\s(t_j)/\zeta_2^\s(t_j)=1$ where $y(u)=(u-t_1)\cdots(u-t_l)$. Here we drop the upper indices since the rank is 1. Instead of using this form, we assume that the Bethe ansatz equation is equivalent to the condition that the polynomial $y(u)$ divides the polynomial $\varphi^\s(u)-\psi^\s(u)$, see e.g. \cite{HLM:2019,LM:2019}. Hence there exists another monic polynomial $\tl y(u)$ satisfying 
\beq\label{eq:fermionic-repro}
y(u)\tl y(u+s_1)\propto (\varphi^\s(u)-\psi^\s(u)).
\eeq 
By \cite[Lemma 4.3]{HLM:2019}, the monic polynomial $\tl y(u)$ corresponds to a solution of the Bethe ansatz equation associated to the same finite-dimensional irreducible module\footnote{Note that in \cite[Lemma 4.3]{HLM:2019}, it was shown for the case of tensor products of evaluation polynomial modules. But it can be trivially generalized to any finite-dimensional $\Yone$-modules since all finite-dimensional irreducible $\Yone$-modules are essentially tensor products of evaluation modules.} with the parity sequence $\tl\s$. Moreover, the Bethe vector corresponding to this solution is an eigenvector of the first transfer matrix with the same eigenvalue corresponding to the Bethe vector for $y(u)$, see \cite[Lemma 4.4]{HLM:2019}. According to \cite[Corollary 6.13]{LM:2019}, the action of the universal rational difference operator for $\Yone$ is essentially determined by the action of the first transfer matrix. Therefore, we have
\begin{align}
\Big(1-\zeta_1^{\s}(u)\frac{y(u-s_1)}{y(u)}e^{-\pa_u}\Big)^{s_1}&\Big(1-\zeta_2^{\s}(u)\frac{y(u+s_2)}{y(u)}e^{-\pa_u}\Big)^{s_2} \label{eq:oper-one}\\
=\ &\Big(1-\zeta_1^{\tl\s}(u)\frac{\tl y(u-\tl s_1)}{\tl y(u)}e^{-\pa_u}\Big)^{\tl s_1}\Big(1-\zeta_2^{\tl\s}(u)\frac{\tl y(u+\tl s_2)}{\tl y(u)}e^{-\pa_u}\Big)^{\tl s_2},\nonumber
\end{align}
cf. \cite[Lemma 4.5]{HLM:2019}. Now a direct computation using \eqref{eq:fermionic-repro} and \eqref{eq:oper-one}, cf. \cite[Lemma 2.2]{HLM:2019}, implies that
\[
\zeta_1^{\tl\s}(u)=\zeta_2^\s(u)\frac{\psi^\s(u-s_1)}{\psi^\s(u)},\qquad \zeta_2^{\tl\s}(u)=\zeta_1^\s(u)\frac{\varphi^\s(u-s_1)}{\varphi^\s(u)}.
\]
which are exactly the formulas in Theorem \ref{thm:11}.

\subsubsection{The case of $\gl_{m|n}$}
For the general $\gl_{m|n}$ case, let $\s$ be a parity sequence in $S_{m|n}$. Fix $i\in I$ and assume that $s_i\ne s_{i+1}$. Let $\tl\s$ be the parity sequence in $S_{m|n}$ obtained by switching $s_i$ and $s_{i+1}$. Let $\bm\zeta^{\s}=(\zeta^\s_i(u))_{i\in \bar I}$ be an $\ell^\s$-weight and $L(\bm\zeta^{\s})$ the corresponding irreducible highest $\ell^\s$-weight module.

If $\bm y$ satisfies the Bethe ansatz equation associated to $\bm\zeta^{\s}$ and $\s$, then by \cite[Theorem 5.1]{HLM:2019} there exists a monic polynomial $\tl y_i$ such that
\beq\label{eq:fermionic-reprom|n}
y_i(u)\tl y_i(u+s_i)\propto \big(\varphi_i^\s(u)y_{i-1}(u+s_i)y_{i+1}(u)-\psi_i^\s(u)y_{i-1}(u)y_{i+1}(u-s_{i+1})\big),
\eeq
where $\varphi_i^\s(u)$ and $\psi_i^\s(u)$ are relatively prime monic polynomials in $u$ such that
\[
\frac{\varphi_i^\s(u)}{\psi_i^\s(u)}=\frac{\zeta^\s_{i}(u)}{\zeta^\s_{i+1}(u)},
\] 
see \eqref{eq:rat-cond}. If $L(\bm\zeta^\s)$ is finite-dimensional irreducible, then such $\varphi_i^\s(u)$ and $\psi_i^\s(u)$ always exist. Moreover, the new sequence of polynomials $\bm y^{[i]}=(y_1,\dots,\tl y_{i},\dots,y_{m+n-1})$ satisfies the Bethe ansatz equation associated to $\bm\zeta^{\tl\s}$ and $\tl\s$. Furthermore, the Bethe vector corresponding to $\bm y^{[i]}$ is an eigenvector of the first transfer matrix with the same eigenvalue corresponding to the Bethe vector for $\bm y$, see \cite[Lemma 5.5]{HLM:2019}. It is natural to expect that the action of the universal rational difference operator on these two Bethe vectors are the same, namely
\beq\label{eq:operm|n}
\mathfrak D(u, \bm \zeta^\s,\bs y,\bm s)=\mathfrak D(u, \bm \zeta^{\tl\s},\bs y^{[i]},\tl\s),
\eeq
cf. \cite[Theorem 5.3]{HLM:2019}. 

Comparing the linear factorizations of these two operators as in \eqref{eq bae eigenvalue}, then one would suggest that $\zeta_j^{\tl\s}(u)=\zeta_j^\s(u)$ if $j\ne i, i+1$. Assuming this and cancelling common factors in \eqref{eq:operm|n}, the equality \eqref{eq:operm|n} is reduced to
\begin{align*}
&\ \Big(1-\zeta_i^{\s}(u)\frac{y_{i-1}(u+s_i)y_i(u-s_i)}{y_{i-1}(u)y_i(u)}e^{-\pa_u}\Big)^{s_i}\Big(1-\zeta_{i+1}^{\s}(u)\frac{y_{i}(u+s_{i+1})y_{i+1}(u-s_{i+1})}{y_{i}(u)y_{i+1}(u)}e^{-\pa_u}\Big)^{s_{i+1}}\\
=&\ \Big(1-\zeta_i^{\tl\s}(u)\frac{y_{i-1}(u+\tl s_i)\tl y_i(u-\tl s_i)}{y_{i-1}(u)\tl y_i(u)}e^{-\pa_u}\Big)^{\tl s_i}\Big(1-\zeta_{i+1}^{\tl\s}(u)\frac{\tl y_{i}(u+\tl s_{i+1})y_{i+1}(u-\tl s_{i+1})}{\tl y_{i}(u)y_{i+1}(u)}e^{-\pa_u}\Big)^{\tl s_{i+1}}.
\end{align*}
Similarly, using \eqref{eq:fermionic-reprom|n}, a direct computation (or immediately from the computation of the case of $\gl_{1|1}$ with proper substitutions) implies the equality \eqref{eq:tranmn}. Hence again the transition rule obtained from the fermionic reproduction procedure of Bethe ansatz equation is compatible with the transition rule obtained from representation theory. 

\begin{rem}
The fermionic reproduction procedure of Gaudin-type Bethe ansatz equations associated with orthosymplectic Lie superalgebras has been introduced in \cite{LM:2021} which turns out to be compatible with the odd reflections for orthosymplectic Lie superalgebras. We expect that similar fermionic reproduction procedure of XXX-type Bethe ansatz equations associated with orthosymplectic Lie superalgebras will also be compatible with the odd reflections for orthosymplectic Yangians, cf. \cite{AAC:2003,M:2021b}. \qed
\end{rem}

\subsection{A discussion on change of $q$-characters}\label{sec:alg}
Unlike the $\glMN$ case, where the Borel subalgebra varies under the odd reflections while the Cartan subalgebra remains the same, there is no need to discuss the change of characters for $\glMN$-modules under the odd reflections. However, the Gelfand-Tsetlin subalgebra do depends on the parity sequence $\s$. Hence it is natural to consider how the $q$-character of a finite-dimensional $\YglMN$-module $M$ changes with respect to odd reflections.

Given $i\in I$, let $\s$ be a parity sequence in $S_{m|n}$ such that $s_i\ne s_{i+1}$. Let $\tl\s$ be the parity sequence in $S_{m|n}$ obtained by switching $s_i$ and $s_{i+1}$. 

We need the following well known lemma, see e.g. \cite[Proposition 3.4]{L:2021a} for a proof in the super case. For each $i\in I$ and $a\in\C$, define the \emph{simple $\ell^\s$-root} $A^\s_{i,a}\in \mathfrak B$ by
\[
(A^\s_{i,a})_j(u)=\frac{u-a}{u-a-(\alpha_i,\epsilon_j)},\qquad j\in \bar I.
\]
\begin{lemma}\label{lem:affineroot}
Let $M$	be a finite-dimensional $\YglMN$-module. Pick and fix any $i\in I$. Let $(\bm \mu,\bm \nu)$ be a pair of $\ell$-weights of $M$.
\begin{enumerate}
\item  If $e_{i,r}(M_{\bm\mu})\cap M_{\bm\nu}\ne \{0\}$ for some $r\gge 1$, then $\bm \nu=\bm\mu A_{i,a}$ for some $a\in \C$.
\item If $f_{i,r}(M_{\bm\mu})\cap M_{\bm\nu}\ne \{0\}$ for some $r\gge 1$, then $\bm \nu=\bm\mu A_{i,a}^{-1}$ for some $a\in \C$.\qed
\end{enumerate}
\end{lemma}

Like the previous discussions, the problem is reduced to the $\Yone$ case. Let $m=n=1$, define a partial order on $\ell^\s$-weights. Let $(\bm \mu,\bm \nu)$ be a pair of $\ell^\s$-weights. We say that $\bm\nu >^\s \bm \mu$ if $\bm\nu =\bm\mu A_{a_1}^\s\cdots A_{a_r}^\s$ for some $r\in \Z_{>0}$ and $a_j\in \C$, $1\lle j\lle r$. Note that in this case $(\alpha_1,\epsilon_1)=(\alpha_1,\epsilon_2)=s_1$.

Give a finite-dimensional irreducible $\Yone$-module $M$ and suppose we know explicitly its $q$-character $\chi^\s(M)$.  Let
\[
\mathscr Q^\s=\chi^\s(M)=\sum_{\bm\xi}a_{\bm\xi}[\bm\xi],\quad \mathscr Q^{\tl\s}=0
\]
where $a_{\bm \xi}\in\Z_{\gge 0}$. We give an algorithm to compute the $q$-character $\chi^{\tl\s}(M)$ after the odd reflections as follows.
\begin{enumerate}
\item Pick a maximal $\ell^\s$-weight $\bm\zeta =(\zeta_1(u),\zeta_2(u))$ with respect to the partial order $>^\s$ such that $a_{\bm\zeta}>0$. Let $\varphi(u)$ and $\psi(u)$ be coprime monic polynomials such that $\zeta_1(u)/\zeta_2(u)=\varphi(u)/\psi(u)$ and set $k=\deg\varphi(u)$.
\item Update $\mathscr Q^\s$ and $\mathscr Q^{\tl\s}$ by the rule:
\[
\mathscr Q^\s=\mathscr Q^\s-\sum_{J\subset \{1,\dots,k\}}[\bm\zeta]\cdot\left[\Big(\frac{\varphi_{J}(u-s_1)}{\varphi_{J}(u)},\frac{\varphi_{J}(u-s_1)}{\varphi_{J}(u)}\Big)\right],
\]
\[
\mathscr Q^{\tl\s}=\mathscr Q^{\tl\s}+\sum_{J\subset \{1,\dots,k\}}\left[\Big(\zeta_2(u)\frac{\psi(u-s_1)}{\psi(u)},\zeta_1(u)\frac{\varphi(u-s_1)}{\varphi(u)}\Big)\right]\cdot \left[\Big(\frac{\psi_J(u)}{\psi_J(u-s_1)},\frac{\psi_J(u)}{\psi_J(u-s_1)}\Big)\right],
\]
see Lemma \ref{lem:char-1-1} and Theorem \ref{thm:11}. Here $\varphi_J(u)$ and $\psi_J(u)$ are defined as in Section \ref{sec:odd11}.
\item Repeat Steps 1 and 2 until $\mathscr Q^\s=0$.
\end{enumerate}
Then $\chi^{\tl\s}(M)=\mathscr Q^{\tl\s}$.
\begin{rem}
In combination with the algorithm for even case from \cite[Section 5.5]{FM:2001}, we expect that Lemma \ref{lem:char-1-1} can be used to construct $q$-characters for super Yangian $\YglMN$.\qed
\end{rem}

\section{Skew representations}\label{sec skew}
Let $m'$ and $n'$ be nonnegative integers. Let $\s'\in S_{m'|n'}$. Consider the embedding of $\gl_{m'|n'}^{\s'}$ into $\gl_{m'+m|n'+n}^{\s'\s}$ sending $e_{ij}^{\s'}$ to $e_{ij}^{\s'\s}$ for $1\lle i,j\lle m'+n'$. 

Let $\la$ and $\mu$ be an $(m'+m|n'+n)$-hook partition and an $(m'|n')$-hook partition, respectively. Suppose further that $\la_i\gge \mu_i$ for all $i\in \Z_{>0}$. Consider the skew Young diagram $\la/\mu$.

Let $\mu^{\s'}$ be the $\gl_{m'|n'}^{\s'}$-weight corresponding to $\mu$, and let $\la^{\s'\s}$ be the $\gl_{m'+m|n'+n}^{\s'\s}$-weight corresponding to $\la$. We have the finite-dimensional irreducible $\gl_{m'+m|n'+n}^{\s'\s}$-module $L(\la^{\s'\s})$.  Consider $L(\la^{\s'\s})$ as a $\gl_{m'|n'}^{\s'}$-module. In the following, we shall drop $\s$ and $\s'$ for brevity.

Define $L(\la/\mu)$  to be the subspace of $L(\la)$ by
\[
L(\la/\mu):=\{v\in L(\la)~|~e_{ii}v=\mu(e_{ii})v,\ e_{jk}v=0,\, 1\lle i\lle m'+n',\ 1\lle j<k\lle m'+n'\}.
\]
The subspace $L(\la/\mu)$ has a natural $\mathrm{U}(\gl_{m'+m|n'+n})^{\gl_{m'|n'}}$-module structure.

Recall $\psi_{m'|n'}$ from \eqref{eq:psi}. Regard $\mathrm{Y}(\gl_{m'|n'})$ as the subalgebra of $\mathrm Y(\gl_{m'+m|n'+n})$ via the natural embedding $t_{ij}(u)\mapsto t_{ij}(u)$ for $1\lle i,j\lle m'+n'$. 

\begin{lemma}[{\cite[Lemma 4.3]{P:2016}}]\label{lem commute yangian}
The subalgebra $\mathrm{Y}(\gl_{m'|n'})$ of $\mathrm Y(\gl_{m'+m|n'+n})$ supercommutes with the image of $\YglMN$ under the map $\psi_{m'|n'}$.
\end{lemma}
Recall the evaluation map $\pi$, see \eqref{eq:evaluation-map}. The following is straightforward from Lemma \ref{lem commute yangian}. 

\begin{corollary}\label{prop centralizer}
The image of the homomorphism $$\pi_{m'+m|n'+n}\circ \psi_{m'|n'}:\YglMN\to \mathrm{U}(\gl_{m'+m|n'+n})$$ supercommutes with the subalgebra $\mathrm{U}(\gl_{m'|n'})$ in $\mathrm{U}(\gl_{m'+m|n'+n})$.
\end{corollary}

Corollary \ref{prop centralizer} implies that the subspace $L(\la/\mu)$ is invariant under the action of the image of $\pi_{m'+m|n'+n}\circ \psi_{m'|n'}\circ \tau_{m'-n'}$, where $\tau_{m'-n'}$ is the shift automorphism defined in \eqref{eq tau z}. Therefore, $L(\la/\mu)$ is a $\YglMN$-module. We call $L(\la/\mu)$ a {\it skew representation}, cf. \cite{Che:1987}.

We study the $q$-character of the skew representation  $L(\la/\mu)$ of $\YglMN$ in the rest of this section. 

For the parity sequence $\s$, define $\kappa_1=0$ if $s_1=1$ and $\kappa_1=-1$ if $s_1=-1$. For $2\lle i\lle m+n$, define inductively $\ka_i$ by the rule: $\ka_i=\ka_{i-1}$ if $s_i\ne s_{i-1}$ and $\ka_i=\ka_{i-1}+s_i$ if $s_i= s_{i-1}$.

For each $a\in \C$ and $i\in \bar I$, let
\[
\mathscr X_{i,a}= \left(1,\dots, \big(1+(u+a+\kappa_i)^{-1}\big)^{s_i}, \dots,1\right) \in \mathfrak B.
\]
Here the only component not equal to 1 is at the $i$-th position.

The elements in $A=\{1,\dots,m,\underline 1,\dots,\underline n\}$ are rearranged to $(\mathfrak a_1,\dots,\mathfrak a_{m+n})$ as in Section \ref{sec:hook-p}.
For a semi-standard Young $\s$-tableau $\mathcal T$ of shape $\la/\mu$, recall that $\mathcal T(i,j)$ denotes the index $k$ if the element $\mathfrak a_k$ is in the box representing the pair $(i,j)\in \la/\mu$ and $c(i,j)=j-i$ is the content of the pair $(i,j)$.

It is known from \cite{CPT:2015} that the dimension of $L(\la/\mu)$ is equal to the number of semi-standard Young $\s$-tableaux of shape $\la/\mu$. The following theorem is a refinement of this statement which is a super analog of \cite[Lemma 2.1]{NT:1998}, see also \cite[Theorem 3.4]{LM:2020} for the case of standard parity sequence.
\begin{theorem}\label{thm character}
The $q$-character of the $\YglMN$-module $L(\la/\mu)$ is given by
\[
\mathscr{K}_{\la/\mu}(u):=\chi(L(\la/\mu))=\sum_{\mathcal T}\prod_{(i,j)\in\la/\mu}\mathscr X_{\mathcal T(i,j),c(i,j)},
\]
summed over all semi-standard Young tableaux $\mathcal T$ of shape $\la/\mu$. In particular, $L(\la/\mu)$ is thin.
\end{theorem}
Before proving the theorem, we recall the following proposition from \cite[Theorem 2.43]{Tsy:2020}. Similar to $\kappa_i$, define $\kappa_i'$, $1\lle i\lle m'+n'$, with $m$ and $n$ replaced by $m'$ and $n'$, respectively. Set $\kappa_{m'+n'+j}'=m'-n'+\kappa_j$, $j\in \bar I$. Set $s'_i=s_{i-m'-n'}$, $m'+n'+1\lle i \lle m'+n'+m+n$. Note that $\ka_i'$ for $1\lle i\lle m'+n'+m+n$ are related to the parity sequence $\s'\s$ just as $\ka_i$ for $1\lle i\lle m+n$ are related to $\s$. 
\begin{proposition}[\cite{Gow:2007,Tsy:2020}]\label{prop center}
The coefficients of the series $\prod_{i\in \bar I}(d_i(u-\kappa_i))^{s_i}$ are central in $\YglMN$. The coefficients of the series $\prod_{i=1}^{m'+n'+m+n}(d_i(u-\kappa'_i))^{s'_i}$ are central in $\mathrm{Y}(\gl_{m'+m|n'+n})$.
\end{proposition}
According to \cite[Proposition 3.6]{HM:2020}, if we identify super Yangians with different parity sequences as in Section \ref{sec:id}, then the series $\prod_{i\in \bar I}(d_i(u-\kappa_i))^{s_i}$ is independent of the parity sequence $\s$.
\begin{lemma}\label{lem center action}
Let $\la$ be a Young diagram. Then the operator $\prod_{i\in \bar I}(d_i(u-\kappa_i))^{s_i}$ acts on the evaluation $\YglMN$-module $L(\la^\s)$ by the scalar operator
\[
\prod_{i\in \bar I}(d_i(u-\kappa_i))^{s_i}\Big|_{L(\la^\s)}=\prod_{i\in \bar I}\Big(1+\frac{(\la_i^\s,\epsilon_i)}{u-\kappa_i}\Big)^{s_i}=\prod_{(i,j)\in \la}\frac{u+c(i,j)+1}{u+  c(i,j)}.
\]
Similarly, the operator $\prod_{i=1}^{m'+n'+m+n}(d_i(u-\kappa'_i))^{s'_i}$ acts on the evaluation $\mathrm{Y}(\gl_{m'+m|n'+n})$-module $L(\la^{\s'\s})$ by the scalar operator
\beq\label{eq:ber-action-2}
\prod_{i=1}^{m'+n'+m+n}(d_i(u-\kappa'_i))^{s'_i}\Big|_{L(\la^{\s'\s})}=\prod_{(i,j)\in \la}\frac{u+c(i,j)+1}{u+c(i,j)}.
\eeq
\end{lemma}
\begin{proof}
The statement follows from Proposition \ref{prop center} and direct computations on highest $\ell$-weight vector.
\end{proof}

\begin{proof}[{Proof of Theorem} \ref{thm character}]
The proof of the first statement is similar to \cite[Lemma 2.1]{NT:1998} and \cite[Lemma 4.7]{FM:2002}. We sketch the proof following the exposition of \cite[Corollary 8.5.8]{Molev:2007}. 

We first show it for the case when $m'+n'=0$. Then $L(\la/\mu)$ is indeed the evaluation module $L(\la^\s)$. For a semi-standard Young $\s$-tableau $\mc T$, denote by $\mc T_k$ the sub-tableau consisting of all boxes occupied by elements $\mathfrak a_1,\dots,\mathfrak a_k$ for $k=0,1,\dots,m+n$. We have $\varnothing =\mc T_0\subset\mc T_1\subset \mc T_2\subset \cdots\subset \mc T_{m+n}=\mc T$. Moreover, $\mc T$ is uniquely determined by the data $(\mc T_1,\dots,\mc T_{m+n})$. The $\glMN$-module $L(\la^\s)$ has a basis $\{v_{\mc T}\}$ indexed by semi-standard Young $\s$-tableaux of shape $\la$ which is called the Gelfand-Tsetlin basis and is obtained from the branching rule of $\glMN$, see \cite{BR:1987} for the standard parity sequence and \cite{CPT:2015} for arbitrary parity sequences. Using the properties of Gelfand-Tsetlin basis and Lemma \ref{lem center action}, we obtain that
\beq\label{eq:111}
(d_k(u-\kappa_k))^{s_k}\ v_{\mc T}=\prod_{(i,j)\in \mc T_{k}/\mc T_{k-1}}\frac{u+c(i,j)+1}{u+c(i,j)},
\eeq
completing the proof in the case $m'+n'=0$. In the skew case namely $m'+n'>0$, we use \eqref{eq:ber-action-2} and Lemma \ref{eq psi map}. Then formulas \eqref{eq:111} remain valid as well. Note the shift automorphism $\tau_{m'-n'}$ in the definition of skew representations.

The second statement follows from the fact that different semi-standard Young $\s$-tableaux $\mathcal T$ of the same shape correspond to different $\ell$-weights. Indeed, the data $(\mc T{(i,j)}, c(i,j))$ determine the semi-standard Young $\s$-tableau uniquely, since the content (the second component of the pair) tells us which diagonal it belongs to, and on the same diagonal the elements $\mathfrak a_{\mc T(i,j)}$ occupying the boxes strictly increase.
\end{proof}

\begin{proposition}\label{prop iso}
The skew representation $L(\la/\mu)$ is independent of the choice of parity sequences. In particular, it is irreducible.
\end{proposition}
\begin{proof}
Let $\s$ be a parity sequence in $S_{m|n}$. Fix $i\in I$ and assume that $s_i\ne s_{i+1}$. Let $\tl\s$ be the parity sequence in $S_{m|n}$ obtained by switching $s_i$ and $s_{i+1}$. Let $\sigma$ and $\tl\sigma$ be the simple permutation $(i,i+1)$ and $(m'+n'+i,m'+n'+i+1)$ in the symmetric group $\mathfrak S_{m+n}$ and $\mathfrak S_{m'+m|n'+n}$, respectively. Then the Lie superalgebras $\gl_{m'+m|n'+n}^{\s'\s}$ and $\gl_{m'+m|n'+n}^{\s'\tl\s}$ are isomorphic via the map $e_{ab}^{\s'\s}\mapsto e_{\tl\sigma(a)\tl\sigma(b)}^{\s'\tl\s}$. Moreover, we identify these two Lie superalgebras under this isomorphism, the modules $L(\la^{\s'\tl\s})$ and $L(\la^{\s'\s})$ can be identified as the same $\gl_{m'+m+n'+n}$-module, see e.g. \cite[Section 2.4]{CW:2012}. The first statement follows from the equality
\[
\pi_{m'+m|n'+n}^{\s'\s}\circ \psi_{m'|n'}^{\s}(t_{ab}^{\s}(u))=\pi_{m'+m|n'+n}^{\s'\tl\s}\circ \psi_{m'|n'}^{\tl\s}(t_{\sigma(a)\sigma(b)}^{\tl\s}(u))
\]
which clearly follows from \eqref{eq:psi t}.

Now the second statement follows from the the corresponding statement for the skew representations associated to the standard parity sequence, see \cite[Theorem 4.9]{LM:2020} or \cite[Theorem 3.2]{L:2021a}.
\end{proof}

\end{document}